\documentclass[reqno]{amsart}
\usepackage{bbm}
\usepackage{amsfonts}
\usepackage{mathrsfs}
\usepackage{hyperref}
\usepackage{amssymb}
 \newtheorem{thm}{Theorem}[section]
 \newtheorem{cor}[thm]{Corollary}

 \newtheorem{lem}[thm]{Lemma}
 \newtheorem{ex}[thm]{Example}

 \theoremstyle{remark}
 
 \numberwithin{equation}{section}

\begin{document}
\title[\hfil  On the subset sum problem over finite fields]
{On the subset sum problem over finite fields}

\author{Jiyou Li}
\address{School of Mathematical Sciences, Peking University, Beijing, P.R. China}
\email{joe@math.pku.edu.cn}

\author{Daqing Wan}
\address{Department of Mathematics, University of California, Irvine, CA 92697-3875, USA}
\email{dwan@math.uci.edu}

\thanks{This research is partially supported by the NSFC (10331030).}

\begin{abstract}

The subset sum problem over finite fields is a well-known {\bf NP}-complete problem.
It arises naturally from decoding generalized Reed-Solomon codes.
In this paper, we study the number of solutions of the subset sum problem from a
mathematical point of view. In several interesting cases, we obtain
explicit or asymptotic formulas for the solution number.
As a consequence, we obtain some results on
the decoding problem of Reed-Solomon codes.
\end{abstract}

\maketitle
 \numberwithin{equation}{section}
\newtheorem{theorem}{Theorem}[section]
\newtheorem{lemma}[theorem]{Lemma}
\newtheorem{example}[theorem]{Example}
\allowdisplaybreaks

\section{Introduction}

Let ${\bf F}_q$ be a finite field of characteristic $p$. Let
$D\subseteq {\bf F}_q$ be a subset of cardinality $|D|=n>0$. Let
$1\leq m\leq k\leq n$ be integers. Given $m$ elements $b_1,\cdots ,
b_m$ in ${\bf F}_q$. Let $V_{b,k}$ denote the affine variety in ${\bf
A}^k$ defined by the following system of equations
$$\sum_{i=1}^k X_i =b_1,$$
$$\sum_{1\leq i_1<i_2 \leq k}X_{i_1}X_{i_2} =b_2,$$
$$\cdots,$$
$$\sum_{1\leq i_1<i_2<\cdots <i_m\leq k}X_{i_1}\cdots X_{i_k}=b_m,$$
$$X_i-X_j\not=0 ~(i\not=j).$$
A fundamental problem arising from decoding Reed-Solomon codes is to
determine for any given $b=(b_1,\cdots, b_m)\in {\bf F}_q^m$, if the
variety $V_{b,k}$ has an ${\bf F}_q$-rational point with all $x_i\in
D$, see section 5 for more details. This problem is apparently
difficult due to several parameters of different nature involved.
The high degree of the variety naturally introduces a substantial
algebraic difficulty, but this can at least be overcome in some
cases when $D$ is the full field  ${\bf F}_q$ and $m$ is small,
using the Weil bound. The requirement that the $x_i$'s are distinct
leads to a significant combinatorial difficulty. From computational
point of view, a more substantial difficulty is caused by the
flexibility of the subset $D$ of ${\bf F}_q$. In fact, even in the
case $m=1$ and so the algebraic difficulty disappear, the problem is
known to be ${\bf NP}$-complete. In this case, the problem is
reduced to the well known subset sum problem over $D\subseteq {\bf
F}_q$, that is, to determine for a given $b\in {\bf F}_q$, if there
is a non-empty subset $\{x_1,x_2,\cdots,x_k \}\subseteq D$ such that
$$x_1+x_2+\cdots +x_k=b.    \eqno{(1.1) }$$
This subset sum problem is known to be $\mathbf{NP}$-complete.
Given integer $1\leq k \leq n$, and $b\in {\bf F}_q$, a more precise problem is
to determine
$$ N(k,b,D)=\#\{ \{x_1,x_2,\cdots,x_k \}\subseteq D \mid x_1+x_2+\cdots
+x_k=b \},$$
the number of $k$-element subsets of $D$ whose sum is $b$.
The decision version of the above subset sum problem is
then to determine if $N(k,b,D)>0$ for some $k$, that is, if
$$N(b,D):=\sum_{k=1}^n N(k,b,D)>0.$$

In this paper, we study the approximation version of the above
subset sum problem for each $k$ from a mathematical point of view,
that is, we try to approximate the solution number $N(k,b,D)$.
Intuitively, the problem is easier if $D$ is close to be the full
field ${\bf F}_q$, i.e., when $q-n$ is small. Indeed, we obtain an
asymptotic formula for $ N(k,b,D)$ when $q-n$ is small.
Heuristically, $ N(k,b,D)$ should be approximately  $\frac 1 q {n
\choose k} $. The question is about the error term. We have

\begin{thm}\label{thm1.1} Let $p<q$, that is, ${\bf F}_q$ is not a prime field.
Let $D\subseteq {\bf F}_q$ be a subset of cardinality $n$. For any
$1\leq k \leq n \leq q-2$, any $b\in {\bf F}_q$, we have the
inequality
$$\bigg |N(k,b,D)-\frac 1 q {n \choose k} \bigg | \leq
  \frac{q-p}q{k+q-n-2\choose q-n-2}{q/p-1 \choose \lfloor k/p\rfloor}.$$
  Furthermore, let $D={\bf F}_q\backslash{\{a_1,\cdots,a_{q-n}\}}$ with
  $a_1=0$, and if {$b,a_2,\cdots,a_{q-n}$} are linearly independent over  ${\bf F}_p$,
  then we have the improved estimate
$$\bigg |N(k,b,D)-\frac 1 q {n \choose k} \bigg | \leq
 \max_{0 \leq j\leq k}\frac p q\cdot{k+q-n-2-j\choose q-n-2}{q/p-1 \choose \lfloor j/p\rfloor}.$$

When $q=p$, that is, ${\bf F}_q$ is a prime field, we have
$$ \bigg |N(k,b,D)-\frac 1 q {n
\choose k}+\frac{(-1)^k}q{k+q-n-1 \choose q-n-1} \bigg |
 \leq  {k+q-n-2\choose q-n-2}.  $$
\end{thm}
Theorem \ref{thm1.1} assumes that $n\leq q-2$. In the remaining
case $n\geq q-2$, that is, $n\in \{ q-2, q-1, q\}$,  the situation
is nicer and we obtain explicit formulas for $N(k,b,D)$. Here we
first state the  results for $q-n\leq 1$ and thus we can take
$D={\bf F}_q$ or ${\bf F}_q^*$.

\begin{thm}\label{thm1.2} Define $v(b)=-1$ if $b\neq 0$, and $v(b)=q-1$ if $b=0$. Then
$$N(k,b, {\bf F}_q^*) = {1\over q}{q-1\choose k} + (-1)^{k +\lfloor k/p\rfloor}{v(b)\over q}{q/p -1
 \choose \lfloor k/p\rfloor }.$$
If $p\nmid k$, then
$$N(k,b, {\bf F}_q) = {1\over q}{q\choose k}.$$
If $p \mid k$, then
$$N(k, b, {\bf F}_q) = {1\over q}{q\choose k} +(-1)^{k +{k\over p}}{v(b)\over q}{q/p\choose k/p}.$$
\end{thm}
When $q-n=2$, note that we can always take $D={\bf
F}_q\backslash{\{0, 1\}}$.
\begin{thm}\label{thm1.3}
 Let $q>2$. Then we have
 \begin{eqnarray*}
 N(k,b,{\bf F}_q\backslash{\{0, 1\}})=\frac 1 q{q-2 \choose k}+
 \frac 1q(-1)^k R^2_k-(-1)^kS(k,k-b),
 \end{eqnarray*}
  where $R^2_k,S(k,b)$ are defined as in (3.2) and (3.3).
 \end{thm}
This paper is organized as follows:
 We first prove Theorem \ref{thm1.2} and Theorem \ref{thm1.3}
in Section 2 and Section 3 respectively. Then we prove Theorem \ref{thm1.1}
in Section 4.
 Applications to coding theory are given in Section $5$.

{\bf Notations}. For $x\in\mathbb{R}$, let  $(x)_0=1$
 and $(x)_k=x (x-1) \cdots (x-k+1)$
for $k\in $ $\mathbb{Z^+}=\{1,2,3,\cdots\} $. For $k\in
\mathbb{N}=\{0,1,2,\cdots\}$ define the binomial coefficient ${x
\choose k}=\frac {(x)_k}{k!}$. For a real number $a$ we denote
$\lfloor a \rfloor$ to be the largest integer not greater than $a$.

\section{Proof of Theorem \ref{thm1.2}}

When $D$ equals $q-1$, it suffices to consider $N(k,b,{\bf F}_q^*)$
by a simple linear substitution. Let $M(k,b,D)$ denote the number of
ordered tuples $(x_1,x_2,\cdots, x_k)$ satisfying equation (1.1).
Then
$$M(k,b,D)=k!N(k,b,D)$$
is the number of solutions of the equation
\begin{eqnarray}\label{2.1}
{x_1+\cdots +x_k =b, x_i\in D, x_i\not= x_j ~(i\not=j).}
\end{eqnarray}

It suffices to determine $M(k,b,D)$. We use a pure combinatorial
method to find recursive relations among the values of $M(k,b,{\bf
F}_q)$ and $M(k,b,{\bf F}_q^*)$.

 \begin{lem} \label{lem2.1}  For $b\neq 0$ and $D$ being ${\bf F}_q$ or
 ${\bf F}_q^*$, we have $M(k,b, D)=M(k,1,D)$.

 \end{lem}
 \begin{proof}
 There is a one to one map sending the solution
$\{x_1,x_2,\cdots, x_k\}$ of (\ref{2.1})
 to the solution $\{x_1b^{-1},x_2b^{-1},\cdots, x_kb^{-1}\}$ of (\ref{2.1}) with $b=1$.
 \end{proof}

 \begin{lem}\label{lem2.2}
\begin{eqnarray}
\label{2.2}M(k,1,{\bf F}_q)=M(k,1,{\bf F}_q^{*})+k M(k-1,1,{\bf F}_q^{*}),\\
\label{2.3}M(k,0,{\bf F}_q)=M(k,0,{\bf F}_q^{*})+k M(k-1,0,{\bf F}_q^{*}),\\
\label{2.4}(q)_k=(q-1)M(k,1,{\bf F}_q)+M(k,0,{\bf F}_q),\\
\label{2.5} (q-1)_k=(q-1)M(k,1,{\bf F}_q^{*})+M(k,0,{\bf F}_q^{*}).
\end{eqnarray}
\end{lem}

 \begin{proof} Fix an element $c\in {\bf F}_q$. The solutions of (\ref{2.1}) in ${\bf F}_q$ can be divided into two
 classes depending on whether $c$ occurs.
By a linear substitution, the number of solutions of (\ref{2.1}) in ${\bf
F}_q$ not including $c$ equals $M(k,b-ck,{\bf F}_q^{*})$.  And the
number of solutions of (\ref{2.1}) in ${\bf F}_q$ including $c$ equals
$kM(k-1,b-ck,{\bf F}_q^*)$.  Hence we have
\begin{eqnarray}
\label{2.6}M(k,b,{\bf F}_q)=M(k,b-c k,{\bf F}_q^{*})+k M(k-1,b-ck,{\bf F}_q^{*}).
\end{eqnarray}
 Then  (\ref{2.2}) follows by choosing $b=1,c=0$.
 Similarly, (\ref{2.3}) follows by choosing $b=0,c=0$.
 Note that $(q)_k$ is the number of $k$-permutations of ${\bf F}_q$,
and $(q-1)_k$ is the number of $k$-permutations of ${\bf F}_q^{*}$.
Thus, both (\ref{2.4}) and (\ref{2.5}) follows.
 \end{proof}

The next step is to find more relations between $M(k,b,{\bf F}_q)$
and $M(k,b,{\bf F}_q^*)$.

\begin{lem} \label{lem2.3}  If $p\nmid k $, we have $M(k,b,{\bf F}_q)=M(k,0,{\bf F}_q)$ for all $b\in {\bf F}_q$ and hence
$$M(k,b, {\bf F}_q) =\frac{1}{q}(q)_k.$$
If $p\mid k $, we have $M(k,b,{\bf F}_q)=qM(k-1,b,{\bf F}_q^{*})$
for all $b\in {\bf F}_q$.
\end{lem}

 \begin{proof} \textbf{Case 1:} Since $p\nmid k $, we can take $c=k^{-1}b$ in (\ref{2.6}) and get the relation
$$M(k,b,{\bf F}_q)=M(k,0,{\bf F}_q^{*})+k M(k-1,0,{\bf F}_q^{*}).$$
The right side is just $M(k,0,{\bf F}_q)$ by (\ref{2.3}).

\ \ \ \ \textbf{Case 2 :} In this case, $p\mid k $.  Then
$M(k,b,{\bf F}_q)$ equals the number of ordered solutions of the
following system of equations:
$$\left\{
\begin{array}{ll}
  x_1+x_2+\cdots+x_k=b, \ \  \hbox{}\\
  x_1-x_2=y_2, \ \  \hbox{}\\
  \cdots \cdots \ \  \hbox{}\\
  x_1-x_k=y_k, \ \  \hbox{}\\
  y_i \in {\bf F}_q^*, \ \  y_i\not= y_j,  \  2\leq i< j\leq k£¬ .
    \end{array}
    \right.
    $$
   Regarding $x_1,x_2,\cdots, x_k$ as variables it is easy to check
that the $p$-rank (the rank of a matrix over the prime field ${\bf
F}_p$) of the coefficient matrix of the above system of equations
equals $k-1$. The system has solutions if and only if
$\sum_{i=2}^{k}y_i=-b$ and $y_i\in {\bf F}_q^*$ being distinct.
Furthermore, since the $p$-rank of the above
 system is $k-1$, when $y_2,y_3,\cdots, y_k$ and $x_1$ are given
 then $x_2,x_3,\cdots, x_k$ will be uniquely determined. This means the
 number of the solutions of above linear system of equations equals
 to $q$ times the number of ordered solutions of the
 following equation:
$$\left\{
\begin{array}{ll}
  y_2+y_3+\cdots+y_k=-b, \ \  \hbox{}\\
  y_i \in {\bf F}_q^*, \ \ \  y_i\not= y_j,  \  2\leq i< j\leq k£¬ .
    \end{array}
    \right.
    $$
  This number of solutions of the above equation is just $M(k-1,b,{\bf F}_q^{*})$ and
  hence    $M(k,b,{\bf F}_q)=qM(k-1,b,{\bf F}_q^{*})$.
 \end{proof}

We have obtained several relations from Lemma \ref{lem2.2} and Lemma
\ref{lem2.3}.  To determine $M(k,b, {\bf F}_q)$, it is now
sufficient to know $M(k,0,{\bf F}_q^{*})$. Define for $k>0$,
$$d_k=M(k,1,{\bf F}_q^{*})-M(k,0,{\bf
F}_q^{*}).$$ Then by (\ref{2.5}) we have
\begin{eqnarray}\label{2.7}
q M(k,0,{\bf F}_q^{*})=(q-1)_{k}-(q-1)d_k.
\end{eqnarray}
Heuristically, $M(k,0,{\bf F}_q^*)$ should be approximately $\frac 1
q (q-1)_k $. To obtain the explicit value of $M(k,0,{\bf F}_q^{*})$,
we only need to know $d_k$. For convenience we set $d_0=-1$.

\begin{lem}\label{lem 2.4} If $d_k$ is defined as above, then
$$d_{k}=\left\{
\begin{array} {ll}
  -1, \ \ \ \ \ \ \ \  \ \ \ \ \ \ \ \ \ \ \   \hbox{$k=0$};\\
\  1, \ \ \ \ \ \ \ \  \ \ \ \ \ \ \ \ \ \ \ \  \hbox{$k=1$};\\
 -kd_{k-1},\ \  \ \ \ \ \ \ \ \ \   \hbox{$p\nmid k $, \   $2\leq k\leq q-1$ };\\
     (q-k)d_{k-1}, \ \ \ \ \ \  \hbox{$p\mid k$,\   $2\leq k\leq q-1$ }. \\
         \end{array}
    \right.$$
    \end{lem}
\begin{proof}
One checks that
 $d_1=M(1,1,{\bf F}_q^{*})-M(1,0,{\bf F}_q^{*})=1-0=1.$
When $p \nmid k $, by Lemma \ref{lem2.3}  we have
 $M(k,1,{\bf F}_q)=M(k,0,{\bf F}_q).$ This together with Lemma \ref{2.2} implies
$$M(k,1,{\bf F}_q^{*})-M(k,0,{\bf F}_q^{*})=k(M(k-1,0,{\bf F}_q^{*})-M(k-1,1,{\bf F}_q^{*})).$$
Namely, $d_k=-kd_{k-1}.$ When $p\mid k $, using Lemma \ref{lem2.3}
we have
$$M(k, 1, {\bf F}_q) -M(k,0, {\bf F}_q)= q(M(k-1,1,{\bf F}_q^*)-M(k-1,0,{\bf F}_q^*)) =qd_{k-1}.$$
By Lemma \ref{2.2}, the left side is $d_k +kd_{k-1}$. Thus,
$d_k=(q-k)d_{k-1}.$
 \end{proof}

\begin{cor}\label{cor2.5}
$$
d_k=-(-1)^{k+\lfloor k/p\rfloor}k!{q/p-1\choose \lfloor  k/p\rfloor}.
$$
  \end{cor}
  \begin{proof}
     One checks $d_0=-1$ and $d_1=1$ are consistent with
     the above formula for $k\leq 1$.
  Let $k\geq 2$ and write $k=np+m$ with $0\leq m<p$.
By Lemma
 \ref{lem 2.4},
\begin{eqnarray*}
\frac {d_k}{k!} &=&(-1)^{n(p-1)+m+1}\prod_{i=1}^{n}\frac{(q-ip)}{ip}\\
&=&(-1)^{n(p-1)+m+1}\frac{\prod_{i=1}^{n}{(
q/p-i)}}{n!}\\
&=&-(-1)^{k+n}{q/p-1\choose n}.
\end{eqnarray*}
It is easy to check that if $q=p$,
then we have $d_k=(-1)^{k-1}k!$,
 which is consistent with the definition  $(0)_0=1$.
\end{proof}

\textbf{Proof of Theorem \ref{thm1.2}} \
  Let $M(k,b,D)$ be the number of solutions
of (\ref{2.1}). Note that  $M(k,b,D)=k!N(k,b,D)$ and
$d_k=-(-1)^{k+\lfloor k/p\rfloor}k!{q/p-1\choose \lfloor  k/p\rfloor}$.
Thus it is sufficient to prove
 \begin{eqnarray*}
&&M(k,b,{\bf F}_q^{*})=\frac {(q-1)_k - v(b)d_k}{q};\\
&&M(k,b,{\bf F}_q)=\frac {(q)_k- v(b)(d_k+kd_{k-1})}{q}.\\
 \end{eqnarray*}
  If $b=0$, by (\ref{2.7}), we obtain
$$qM(k,0, {\bf F}_q^*) = (q-1)_k -(q-1)d_k.$$
If $b\not=0$, then
$$qM(k, b, {\bf F}_q^*) = qM(k,1, {\bf F}_q^*)=qd_{k}+ qM(k, 0, {\bf F}_q^*) = (q-1)_k +d_k.$$
The formula for $M(k,b, {\bf F}_q^*)$ holds.

If $p\nmid k$, then $d_k +kd_{k-1}=0$ and the formula for $M(k,b,
{\bf F}_q)$ holds by Lemma \ref{2.3}.

If $p\mid k$, then $d_k +kd_{k-1}=qd_{k-1}$. By Lemma \ref{2.3} and the
above formula for $M(k, b, {\bf F}_q^*)$, we deduce
$$M(k, b, {\bf F}_q) = q M(k-1, b, {\bf F}_q^*)= (q-1)_{k-1}-v(b)d_{k-1}.$$
The formula for $M(k,b, {\bf F}_q)$ holds.
The proof is complete.

Now we turn to deciding when the solution number $N(k,b,{\bf
F}_q^*)>0$.
 A sequence $\{ a_0,a_1,\cdots,a_n\}$ is
\textbf{unimodal} if there exits index $k$ with $0 \leq k\leq
n$ such that
$$a_0\leq a_1\leq \cdots a_{k-1}\leq a_k \geq a_{k+1} \cdots\geq a_n.$$
The sequence $\{ a_0,a_1,\cdots,a_n\}$ is called symmetric if $a_i
=a_{n-i}$ for $0\leq i<n$.

\begin{cor}\label{cor2.7}For any $b\in {\bf F}_q$,
both the sequence $N(k,b,{\bf F}_q)$ ($1\leq k\leq q$) and the
sequence $N(k,b,{\bf F}_q^{*})$ ($1\leq k\leq q-1$) are unimodal and
symmetric.
  \end{cor}
  \begin{proof}

The symmetric part can be verified using
     Theorem \ref{thm1.1}. A simpler way is to use the relation
$$\sum_{a\in {\bf F}_q} a = \sum_{a\in {\bf F}_q^*} a =0.$$

To prove the unimodal property for $N(k, b, {\bf F}_q^*)$, by the
symmetry it is sufficient to consider the case $k\leq \frac {q-1}
2$. Then, by Theorem \ref{thm1.1}, we deduce
\begin{eqnarray*}
&&q\left(N(k,0,{\bf F}_q^{*})-N(k-1,0,{\bf F}_q^{*})\right)\\
&&\geq {q-1\choose k}-
{q-1\choose k-1}-(q-1)\left({q/p-1\choose
\lfloor k/p\rfloor}-{q/p-1\choose \lfloor (k-1)/p\rfloor}\right).
\end{eqnarray*}
If $p\nmid k$, then $\lfloor k/{p} \rfloor =\lfloor (k-1)/{p} \rfloor$
and the right side is clearly positive. If $p\mid k$, then
\begin{eqnarray}
&&q\left(N(k,0,{\bf F}_q^{*})-N(k-1,0,{\bf F}_q^{*})\right)\nonumber\\
  &&\geq \frac {q-2k} k {q-1\choose k-1}-(q-1)\frac
{q/p- {2k}/p}{{k}/p}{q/p-1\choose k/p-1}\nonumber\\
&&\label{2.8}=\frac {q-2k} k \left({q-1\choose k-1}-(q-1) {q/p-1\choose
k/p-1}\right).
  \end{eqnarray}
When $p=2$ and $k=2, 4$, or $q\leq9$,
it is easy to  checks that ${q-1\choose k-1}\geq(q-1) {q/p-1\choose
k/p-1}$. Otherwise by the Vandermonde's convolution
  $${q-1 \choose k-1 }=\sum_{i=0}^{q/p -1}{q/p-1\choose
 i}{q-q/p\choose k-1-i},$$
it suffices to prove
$${q-q/p\choose k-k/p}\geq q-1.$$
This inequality follows by noting that
$${q-q/p\choose k-k/p}\geq {q/2\choose 2}$$
and $q>9$. Thus $N(k,0,{\bf F}_q^{*})$ is unimodal.
The proof for the unimodality of $N(k,b,{\bf F}_q)$ is
similar. This completes the proof.

\end{proof}

\begin{cor}\label{cor2.8}
Let $|D|=q-1>4$. If $p$ is an odd prime then for $1<k<q-2$ the
equation $(1.1)$ always has a solution. If  $p=2$, then for
$2<k<q-3$ the equation $(1.1)$ always has a solution.
  \end{cor}
\begin{proof}
For any $a\in {\bf F}_q$ we have $N(k,b,{\bf F}_q\backslash{\{a\}})
=N(k,b-ka,{\bf F}_q^*)$. Thus it is sufficient
to consider $N(k,1,{\bf F}_q^{*})$ and $N(k,0,{\bf F}_q^{*})$
by Lemma \ref{lem2.1}.
   When $p$ is odd and $k=2$,
 we have  $N(2,0,{\bf F}_q^*)=\frac 1 q({q-1 \choose 2}+(q-1))=\frac {q-1}2>0$, and
  $N(2,1,{\bf F}_q^*)=\frac 1 q({q-1 \choose 2}-1)=\frac {q-3}2>0$ from Theorem \ref{thm1.2}.
  Then, by the unimodality of $N(k,1,{\bf F}_q^{*})$ and $N(k,0,{\bf F}_q^{*})$, for
  $1<k<q-2$, $N(k,b,{\bf F}_q\backslash{\{a\}})$ must be positive.

    Similarly,  when $p=2$ and $k=3$ we have $N(3,0,{\bf F}_q^*)=
\frac 1 q({q-1 \choose 3}+(q-1)(\frac q 2-1))=\frac {(q-1)(q-2)} 6>0$
    and  $N(3,1,{\bf F}_q^*)=\frac 1 q({q-1 \choose 3}-(\frac q 2-1))=\frac {(q-2)(q-4)}
    6>0$. By the unimodality and symmetry we complete the proof.
\end{proof}
\begin{cor}
Let $D={\bf F}_q$.  If $p$ is an odd prime then the equation $(1.1)$
always has a solution if and only if $0<k<q$. If  $p=2$, then for
$2<k<q-2$ the equation (1.1) always has a solution.
  \end{cor}
\begin{proof}
It is straightforward from Corollary \ref{cor2.8} and Theorem
\ref{thm1.1}.
\end{proof}

\section{Proof of Theorem \ref{thm1.3}}
Before our proof of Theorem \ref{thm1.3},
we first give several  lemmas,
which give some basic formulas for the summands of
sign-alternating binomial coefficients.
\begin{lem}\label{lem 3.1}
Let $k,m$ be integers. Then we have
\begin{eqnarray*}
\sum_{k\leq m}(-1)^k{r \choose k } =(-1)^m{r-1 \choose m}.
\end{eqnarray*}
\end{lem}
\begin{proof}
 It follows by comparing the coefficients of $x^m$ in both sides of
$(1-x)^{-1}(1-x)^r=(1-x)^{r-1}$.
\end{proof}


\begin{lem} \label{lem 3.2} Let $<k>_p$ be the least non-negative residue of $k$ modulo $p$.
For any positive integers $a,k$, we have
$$\sum_{j=0}^{k}-(-1)^{\lfloor j/p\rfloor}{a\choose \lfloor j/p\rfloor}=
-p(-1)^{\lfloor k/p\rfloor} {a-1\choose \lfloor k/p\rfloor}
+(p-1-<k>_p)(-1)^{\lfloor k/p\rfloor}{a\choose \lfloor
k/p\rfloor},$$ and thus
\begin{eqnarray}\label{3.1}
  \sum_{j=0}^{k}-(-1)^{\lfloor j/p\rfloor}{a \choose \lfloor
j/p\rfloor}  \leq p{a\choose \lfloor k/p\rfloor}.
\end{eqnarray}
\end{lem}
\begin{proof}
Let $j=n_jp+m_j$ with $0\leq m_j<p$. Applying Lemma \ref{lem 3.1} we
have
\begin{eqnarray*}
&&\sum_{j=0}^{k}-(-1)^{\lfloor \frac j p\rfloor}{a\choose \lfloor
j/p\rfloor}\\
&&=-p\sum_{n_j=0}^{n_k}(-1)^{n_j}{a\choose n_j}+(p-1-<k>_p)(-1)^{n_k}{a\choose n_k}\\
&&=-p(-1)^{\lfloor k/p\rfloor}{a-1\choose \lfloor k/p\rfloor}
+(p-1-<k>_p)(-1)^{\lfloor k/p\rfloor}{a\choose \lfloor k/p\rfloor}.
\end{eqnarray*}
 The inequality (\ref{3.1}) follows by noting the
  alternating signs  before the two binomial coefficients.
\end{proof}
\begin{lem}\label{lem 3.5}
Let $R^1_k=(-1)^k\frac {d_k}{k!}=-(-1)^{\lfloor k/p\rfloor}{q/p-1\choose \lfloor
k/p\rfloor}$.  Let $<k>_p$ denote the least non-negative residue of
$k$ modulo $p$.  Define $R^{2}_k=\sum_{j=0}^{k}R^{1}_j$. Then we
have  \begin{eqnarray} \label{3.2} R^2_{k}=-p(-1)^{ \lfloor
k/p\rfloor}{q/p-2\choose \lfloor k/p\rfloor}
+(p-1-<k>_p)(-1)^{\lfloor k/p\rfloor}{q/p-1\choose  \lfloor
k/p\rfloor}.
\end{eqnarray}
Moreover, let $b\in {\bf F}_p$.
  Define $\delta_{b,k}=1$ if $<b>_p$ is greater than $<k>_p$
  and $\delta_{b,k}=0$ otherwise. Then we have
\begin{eqnarray} \label{3.3}
S(k,b):=\sum_{{0\leq i\leq k} \atop {i\equiv b(\mod p)}}
R^1_{i} =-(-1)^{\lfloor k/p\rfloor} {q/p-2\choose \lfloor k/p
\rfloor} +\delta_{b,k}(-1)^{\lfloor k/p\rfloor}{q/p-1\choose \lfloor
k/p\rfloor}.
\end{eqnarray}
\end{lem}
\begin{proof}
Note that (\ref{3.2}) is direct from Lemma \ref{lem 3.2} by setting
$a=q/p-1$. Since it is similar to that of Lemma \ref{lem 3.2}, we
omit the proof of (\ref{3.3}).
\end{proof}
 We extend the equation (\ref{3.3}) by defining
$S(k,b)=0$ for $b\not \in{\bf F}_p$ and any integer $k$.
Note that $S(k,b)\leq {q/p-2
\choose \lfloor k/p \rfloor}$.
In the following theorem, we give the accurate
formula for $N(k,b,D)$ when $D={\bf F}_q\backslash{\{a_1,a_2\}}$ and
first note that we can always assume $a_1=0$ and $a_2=1$ by a linear
substitution.

\textbf{Proof of Theorem \ref{thm1.3}} \
Using the simple inclusion-exclusion
sieving method by considering whether $a_2$ appears in the solution
of equation (1.1) we have
\begin{eqnarray*}
&&N(k,b,{\bf F}_q\backslash{\{a_1,a_2\}})\\
&&=N(k,b,{\bf F}_q\backslash{\{a_1\}})-N(k-1,b-a_2,{\bf F}_q\backslash{\{a_1,a_2\}})\\
&&=N(k,b,{\bf F}_q\backslash{\{a_1\}})-(N(k-1,b-a_2,{\bf F}_q\backslash{\{a_1\}})\\
&&\ \ \ -N(k-2,b-2a_2,{\bf F}_q\backslash{\{a_1,a_2\}}))\\
&&\ \ \ \  \cdots \cdots\\
&&=\sum_{i=0}^{k-1}(-1)^iN(k-i,b-ia_2,{\bf F}_q\backslash{\{a_1\}})\\
&&\ \ \ +(-1)^{k}N(0,b-ka_2,{\bf F}_q\backslash{\{a_1,a_2\}}).
\end{eqnarray*}
One checks that the above equation holds if we define $N(0,b,D)$ to
be $1$ if and only if $b=0$ for a nonempty set $D$. Noting that $a_1=0$
we have
$$N(k,b,{\bf F}_q\backslash{\{a_1,a_2\}})=\sum_{i=0}^{k}(-1)^iN(k-i,b-ia_2,{\bf F}_q^*). $$
From Theorem \ref{thm1.1} we have the following formula
$$N(k,b,{\bf F}_q^*)=\frac 1 q{q-1 \choose k}-
 \frac1 q (-1)^k v(b) R^1_k,$$
 where $R^1_k=-(-1)^{\lfloor k/p\rfloor}{q/p-1\choose \lfloor
k/p\rfloor}$, $v(b)=-1$ if $b\neq 0$ and $v(b)=q-1$ if $b=0$.
 Thus
\begin{eqnarray*}
&&N(k,b,{\bf F}_q\backslash{\{a_1,a_2\}})\\
&& =\sum_{i=0}^{k}(-1)^i \left(\frac 1 q{q-1 \choose k-i}-
 \frac1 q (-1)^{k-i} v(b-ia_2) R^1_{k-i}\right).\\
&& =\frac 1 q \left( (-1)^k\sum_{k-i=0}^{k}(-1)^{k-i}{q-1 \choose
k-i}-(-1)^k\sum_{k-i=0}^{k}v(b-i a_2)R^1_{k-i}\right)\\
 &&=\frac 1q\left((-1)^k\sum_{j=0}^{k}(-1)^{j}{q-1 \choose
j}-(-1)^k\sum_{j=0}^{k}v(b-ka_2+ja_2)R^1_j\right)\\
 &&=\frac 1 q \left({q-2\choose
k}-(-1)^k\sum_{j=0}^{k}v(b-ka_2+ja_2)R^1_j\right).
\end{eqnarray*}
The last equality follows from Lemma \ref{lem 3.1}.  Noting that $a_2=1$,
and by the definition of $v(b)$ we have
\begin{eqnarray}
&&N(k,b,{\bf F}_q\backslash{\{a_1,a_2\}})\nonumber \\
 &&\label{3.4}=\frac 1 q {q-2 \choose k}
-\frac 1q(-1)^k\sum_{j=0}^{k}v(b-k+j)R^1_j \\
&& =\frac 1 q {q-2 \choose k}-\frac 1q(-1)^k\sum_{j=0}^{k}(-1)\cdot
R^1_j-\frac 1q (-1)^k\sum_{{0 \leq j\leq k}
\atop b-k+j=0}q\cdot R^1_j\nonumber \\
&&=\frac 1 q {q-2 \choose k}+\frac 1q (-1)^k R_k^2-(-1)^k\cdot
S(k,k-b).\label{3.5}
\end{eqnarray}
The proof is complete.

   Combining (\ref{3.4}), (\ref{3.2}) and (\ref{3.3}) we obtain the following simple solution
number formula compared with those stated in Theorem \ref{thm1.2} and Theorem \ref{thm1.3}.
    \begin{cor}\label{cor3.4}
   If $<k>_p=p-1$ and $b\in{\bf F}_p$,  then we have
$$N(k,b,{\bf F}_q\backslash{\{0,1\}})=\frac 1q {q-2 \choose k}+
(-1)^{k +\lfloor k/p\rfloor}\frac {q-p}q{q/p-2\choose \lfloor k/p
\rfloor}.$$
  \end{cor}
This   shows that the estimate in
Theorem $1.1$ is nearly sharp for $q-n=2$.

\section{Proof of Theorem \ref{thm1.1}}
Let $D={\bf F}_q\backslash{\{a_1,a_2,\cdots a_c\}}$, where
$a_1,a_2,\cdots a_c$ are distinct elements in ${\bf F}_q$. In this
section, based on the explicit formula of $N(k,b,D)$ for $c=2$ given
in Theorem \ref{thm1.3}, we first  obtain a general formula for $c>2$. Then  we
give the proof of Theorem \ref{thm1.1}.  The solution number $N(k,b,{\bf
F}_q\backslash{\{a_1,a_2,\cdots a_c\}})$ is closely related to the
${\bf F}_p$-linear relations among the set $\{a_1,\cdots a_c\}$
which we will see in Lemma \ref{lem4.2}.
 For the purpose of Theorem \ref{thm1.1}'s proof and further
investigations on the solution number $N(k,b,D)$, we first state the
following lemma.
\begin{lem}\label{lem4.1}
Let $R^1_k=-(-1)^{\lfloor k/p\rfloor}{q/p-1\choose \lfloor
k/p\rfloor}$.  For $c>1$ if we define recursively that
 $R^{c}_k=\sum_{j=0}^{k}R^{c-1}_j$, then we have
\begin{eqnarray} \label{4.1}
R^{c}_k=-\sum_{j=0}^{k}(-1)^{\lfloor j/p \rfloor }{k+c-2-j \choose
c-2}{q/p-1 \choose \lfloor j/p \rfloor}.
\end{eqnarray}
 \end{lem}
\begin{proof}
When $c=2$, this formula is just the definition of
$R_k^2$. Assume it is true for some $c\geq2$, then we have
\begin{eqnarray*}
R_k^{c+1}&=&\sum_{i=0}^{k}R_i^c\\
 &=&
\sum_{i=0}^{k}(-1)\cdot\sum_{j=0}^{i}(-1)^{\lfloor j/p \rfloor
}{i+c-2-j
\choose c-2}{q/p-1 \choose \lfloor j/p \rfloor}\\
&=& -\sum_{j=0}^{k}\sum_{i=j}^{k}(-1)^{\lfloor j/p \rfloor
}{i+c-2-j \choose c-2}{q/p-1 \choose \lfloor j/p \rfloor}\\
&=& -\sum_{j=0}^{k}(-1)^{\lfloor j/p \rfloor
}\sum_{i=j}^{k}{i+c-2-j \choose c-2}{q/p-1 \choose \lfloor j/p \rfloor}\\
&=& -\sum_{j=0}^{k}(-1)^{\lfloor j/p \rfloor }{k+c-1-j \choose
c-1}{q/p-1 \choose \lfloor j/p \rfloor}.
\end{eqnarray*}
The last equality follows from
 the following simple binomial coefficient identity
   $$\sum_{j\leq k}{j+n\choose n}={k+n+1\choose n+1}.$$
\end{proof}
   It is easy to check that when $k>\frac{q-c}2$,
we have $$N(k,b,D)=N(q-c-k,-b-\sum_{i=1}^{c}{a_i}, D), $$ where
$D={\bf F}_q\backslash{\{a_1,a_2,\cdots,a_c\}}$.
  Thus we may always assume that $k\leq \frac{q-c}2$.
  In the following lemma, for convenience
we state two different types of formulas.
\begin{lem}\label{lem4.2}
Let $D={\bf F}_q\backslash{\{a_1,a_2,\cdots,a_c\}}$ and $c\geq 3$,
where $a_1=0, a_2=1, a_3, \cdots, a_c$ are distinct elements in the
finite field ${\bf F}_q$ of characteristic $p$.  Define the integer
valued function $v(b)=-1$ if $b\neq 0$ and $v(b)=q-1$ if $b=0$. Then
for any $b\in {\bf F}_q$,
we have the formulas
 \begin{eqnarray}
&&\label{4.2}\!\!\!\!\!\!\!\!\!\!\!\!\!\!\!\!\!\!\!\!\!\!\!\!\!\!\!\!\!\!
N(k,b,D)-\frac 1 q{q-c
\choose k}\nonumber \\
 &&\!\!\!\!\!\!\!\!\!\!\!\!\!\!\!\!\!\!\!\!\!\!\!\!\!\!\!\!\!\!= -\frac 1q(-1)^k\cdot\sum_{i_1=0}^k\sum_{i_2=0}^{k-i_1} \cdots
\sum_{i_{c-1}=0}^{k-i_1-\cdots-i_{c-2}}
 v(b-i_1a_c-\cdots-(k-\sum_{j=1}^{c-1}i_{j})a_2)R^1_j   \\
&&\!\!\!\!\!\!\!\!\!\!\!\!\!\!\!\!\!\!\!\!\!\!\!\!\!\!\!\!\!\!=\frac
1q(-1)^kR^c_k-(-1)^k \cdot \sum_{i_1=0}^k\sum_{i_2=0}^{k-i_1} \cdots
\sum_{i_{c-2}=0}^{k-i_1-\cdots-i_{c-3}} \nonumber \\
&&\label{4.3} S(k-\sum_{j=1}^{c-2}i_{j}
,k-\sum_{j=1}^{c-2}i_{j}-b+\sum_{j=1}^{c-2}i_{j}a_{c+1-j}),
 \end{eqnarray}
 where $R_k^c$ is defined by (\ref{3.2}), and $S(k,b)$ is defined by (\ref{3.3}).
 Moreover, if $a_1=0$, and $b, a_2,\cdots,a_c$
  are linear independent over
 ${\bf F}_p$, then we have
 \begin{eqnarray}\label{4.4}
  N(k,b,D)
= \frac1q{q-c \choose k}+\frac1q(-1)^kR^c_k.
 \end{eqnarray}
\end{lem}
\begin{proof}
 Using the simple inclusion-exclusion sieving method we have
\begin{eqnarray*}
&&N(k,b,{\bf F}_q\backslash{\{a_1,a_2,\cdots,a_c\}})\\
&&=N(k,b,{\bf F}_q\backslash{\{a_1,a_2,\cdots,a_{c-1}\}})\\
&&\quad -N(k-1,b-a_c,{\bf F}_q\backslash{\{a_1,a_2,\cdots,a_{c-1}\}})\\
 &&=\cdots \cdots\\
&&=\sum_{i=0}^{k}(-1)^iN(k-i,b-ia_c,{\bf F}_q\backslash\{
a_1,a_2\cdots,a_{c-1}\}).
 \end{eqnarray*}
 When $c=3$, noting that $a_2=1$,  (\ref{3.5}) implies that
 \begin{eqnarray*}
&&N(k,b,{\bf F}_q\backslash{\{a_1,a_2,a_3\}})\\
&=&\sum_{i=0}^{k}(-1)^i \left(\frac 1 q{q-2 \choose k-i} +
 \frac1 q (-1)^{k-i}R_{k-i}^2-(-1)^{k-i}S(k-i,k-i-(b-ia_3))\right)\\
&=&\frac 1q{q-3\choose k}+\frac 1q(-1)^kR_k^3 -(-1)^k\sum _{i=0}^k
S(k-i,k-i-b+ia_3).
\end{eqnarray*}
 By induction, (\ref{4.3}) follows for $c\geq 3$. Similarly, (\ref{4.2}) follows from
 (\ref{3.4}).

 If $b, a_2=1, a_3 \cdots, a_c$ are linear independent over
 ${\bf F}_p$, then first note that $b\not\in{\bf F}_p$.
Thus, when $c=2$, by its extended definition we have $S(k,k-b)=0$ for any
integer $k$.
 When $c>2$, since  $b, a_2=1,  a_3 \cdots, a_c$ are independent,
 we know that $k-\sum_{j=1}^{c-2}i_{j}-b+\sum_{j=1}^{c-2}i_{j}a_{c+1-j}
  \not\in{\bf F}_p$ for any index tuple $(i_1,i_2,\cdots,i_{c-2})$ in the summation of
(\ref{4.3}).
  Thus this summation always vanishes for any $c$ and
   the proof is complete.
\end{proof}
 Now we have obtained the two formulas of the solution number
 $N(k,b,D)$. It suffices to evaluate $R_k^c$ and the summation
in (\ref{4.3}), which is denoted by  $S_k^c$ .
  Unfortunately,  $S_k^c$ is extremely complicated when $c$ is large. The {\bf
NP}-hardness of the subset sum problem indicates the hardness of
precisely evaluating it.  In the following lemmas we first
deduce a simple bounds for $R_k^c$  and $S_k^c$.
\begin{lem} \label{lem4.3}
Let $p<q$. Let
$$S_k^c=\sum_{i_1=0}^k\sum_{i_2=0}^{k-i_1} \cdots
\sum_{i_{c-2}=0}^{k-i_1-\cdots-i_{c-3}} S(k-\sum_{j=1}^{c-2}i_{j}
,k-\sum_{j=1}^{c-2}i_{j}-b+\sum_{j=1}^{c-2}i_{j}a_{c+1-j}).$$
 Then we have
 \begin{eqnarray}\label{4.5}
qS_k^c-R_k^c \leq (q-p){k+c-2\choose c-2}{q/p-1 \choose \lfloor k/p
\rfloor}.
 \end{eqnarray}
\end{lem}
\begin{proof} By the definition of $R_k^c$ and the proof of Lemma
\ref{lem4.1} we have
$$R_k^c=\sum_{i_1=0}^k\sum_{i_2=0}^{k-i_1} \cdots
\sum_{i_{c-2}=0}^{k-i_1-\cdots-i_{c-3}}
R^2(k-\sum_{j=1}^{c-2}i_{j}),$$ where $R^2(k)=R_k^2$.
 From (3.2) and (3.3)
it is easy to check that
$$R_k^2-qS(k,b)\leq (q-p){q/p-1\choose \lfloor k/p \rfloor}$$
for any $b\in{\bf F}_q$  when $p<q$.  Therefore (\ref{4.5}) follows
 since  both the two numbers of terms appear in the two summations
of $R_k^c$ and $S_k^c$ are ${k+c-2\choose c-2}$.
\end{proof}
 Next we turn to giving a bound for $R_k^c$.
 Unfortunately,  even though $R_k^c$ can be written as a simple
sum involving binomial coefficients, it seems nontrivial to evaluate it precisely.
Using equation (\ref{4.1}) and some combinatorial identities, we can easily obtain the following equality
\begin{eqnarray}\label{4.6}
R_k^c&=&-\sum_{j=0}^{\lfloor k/p\rfloor-1}(-1)^{j}\bigg[{k+c-1-ip
\choose c-1}- {k+c-1-ip-p \choose c-1}\bigg]{q/p-1 \choose
j}\nonumber \\
 && +{<k>_p+c-1 \choose c-1}{q/p-1 \choose \lfloor k/p
\rfloor}.
\end{eqnarray}
  It has been known that the simpler sum
$$\sum_{j=0}^{n}(-1)^{j}{2n-1-3i \choose n-1}{n \choose j },$$
which is the coefficient of $x^n$ in $(1+x+x^2)^n$,
 has no closed form.   That means it cannot
be expressed as a fixed number of
hypergeometric terms. For more details  we refer to (\cite{PW}, p.
160). This fact indicates that $R_k^c$ also has no closed form.
Thus, in the next lemma we just give a bound for $R_k^c$ just using
some elementary combinatorial arguments.

In Section 2 we have defined the unimodality of  a sequence.
 A stronger property than unimodality is
logarithmic concavity. First recall that a function $f$ on the real
line is concave if whenever $x < y$ we have $f((x + y)/2)\geq(f(x) +
f(y))/2$.
 Similarly, a sequence $a_0,a_1\cdots,a_n$ of positive numbers
  is {\bf log concave} if $\log a_i$ is a
concave function of $i$ which is to say that $(\log a_{i-1} + \log
a_{i+1} )/2 \leq \log a_i$.
Thus a sequence
is log concave if $a_{i-1}a_{i+1}\leq a_i^2$.
 Using the properties of logarithmic concavity
we have the following lemma.
\begin{lem}\label{lem4.4}
 \begin{eqnarray}\label{4.7}
 R^c_k\leq p\cdot\max_{0\leq j\leq k}{k+c-2-j\choose c-2}{q/p-1
\choose \lfloor j/p \rfloor} . \end{eqnarray}
\end{lem}
\begin{proof}
It is easy to check that both the two sequences ${k+c-2-j\choose
c-2}$ and ${q/p-1\choose \lfloor j/p \rfloor}$ are log concave on
$j$. Thus the sequence $a_j={k+c-2-j\choose c-2}{q/p-1 \choose \lfloor j/p \rfloor}$ is
also log concave on $j$ by the definition of logarithmic concavity .
Since a log concave sequence must be
unimodal, $\{a_j\}$ is unimodal on $j$. Then we have
\begin{eqnarray*}
R_k^c&=&-\sum_{j=0}^k(-1)^{\lfloor j/p \rfloor
}a_j\\
&=&-\sum_{i=0}^{\lfloor k/p \rfloor}(-1)^ia_{ip}-
\cdots-\sum_{i=0}^{\lfloor k/p \rfloor}(-1)^ia_{ip+<k>_p}\cdots-
\sum_{i=0}^{\lfloor k/p \rfloor-1}(-1)^ia_{ip+p-1}.
\end{eqnarray*}
Thus (\ref{4.7}) follows from  the following simple inequality
$$ \sum_{i=0}^k (-1)^i a_i\leq \max_{0\leq i\leq k}a_i,$$
and the proof is complete.
\end{proof}
\textbf{Proof of Theorem \ref{thm1.1}} \ When $q>p$
 we rewrite (\ref{4.3}) to be
  $$N(k,b,D)
= \frac1q{q-c \choose k}+\frac1q(-1)^k(R^c_k-qM^c_k).$$ Applying
(\ref{4.5}) we obtain
\begin{eqnarray}\label{4.8}
 \left |N(k,b,D)-\frac 1 q
{q-c \choose k} \right| \leq \frac {q-p} q{k+c-2 \choose c-2}{q/p-2
\choose \lfloor
 k/p\rfloor}.
  \end{eqnarray}
  If $a_1=0$, and $b,a_2,\cdots,a_c$ are linear independent over
 ${\bf F}_p$, then $S_k^c=0$ for any $k$.
Thus  from (\ref{4.4}) and Lemma \ref{lem4.4}
 we have the improved bound
 \begin{eqnarray}\label{4.9}\bigg |N(k,b,D)-\frac 1 q {q-c \choose k} \bigg |
 \leq {p\over q}\max_{0\leq j\leq k}{k+c-2-j\choose c-2}{q/p-1 \choose \lfloor j/p
 \rfloor}.
  \end{eqnarray}
  Thus we only need to verify the case $q=p$. When
$q=p$, from Lemma \ref{lem4.1} we have
 $$R^{c}_k=-\sum_{j=0}^{k}{k+c-2-j \choose c-2}=-{k+c-1 \choose c-1}.$$
And $S(k,b)$ equals $0$ or $-1$  by its definition given in Lemma \ref{lem 3.5}.
   Thus from (\ref{4.3}) we deduce that
\begin{eqnarray}\label{4.10}
N(k,b,D)=\frac{{p-c \choose k}-(-1)^k{k+c-1 \choose k}
}{p}+(-1)^kM_k^c
\end{eqnarray}
with $0\leq M_k^c\leq{k+c-2 \choose
k}$. Thus
$$   \left
|N(k,b,D)-\frac 1 q {q-c \choose k}+\frac{(-1)^k}q{k+c-1 \choose
c-1} \right|
 \leq  {k+c-2\choose c-2}.$$
Note that $c=q-n$  and the proof is complete.
\begin{ex}
Choose $p=2, q=128, c=4$ and $k=5$. Then $R_k^c=-6840$.
Let $\omega$ be a primitive
element in ${\bf F}_{128}$. Let
$D=F_{128}\backslash{\{0,\omega,\omega^2,\omega^3\}}$ and $b=1$.
Since $1,\omega,\omega^2,\omega^3$ are linear independent,
(\ref{4.4}) gives that there are  $N=1759038$ solutions of the equation (1.1)
compared with the average number $\frac 1 q{q-c\choose
k}\approx1758985$.
 \end{ex}

 {\bf Remark}.\ \ \   If one obtains better bounds for
$S_k^c$, then we can improve the bound given by (\ref{4.8}).
  However, it is much more complicated to evaluate $S_k^c$  than
$R_k^c$. Let
$$I=\{[i_1,i_2,\cdots,i_{c-2}], 0 \leq i_t\leq
k-\sum_{j=1}^{t-1}i_{j},1\leq t \leq c-2:\ \
b-\sum_{j=1}^{c-2}i_{j}a_{c+1-j} \in {\bf F}_p\}.$$  Simple counting
shows that $0 \leq |I|\leq {k+c-2\choose c-2}.$ In the  proof of
(\ref{4.8}) we use the upper bound $|I|\leq {k+c-2\choose c-2}$ and
in the proof of (\ref{4.4}) it is the special case $|I|=0$. We can
  improve the above bound  if we know more information about
   the cardinality of $I$, which is determined by the set ${b, a_2, \cdots, a_c}$.
     For example,
  if we know more about the rank of the set
$\{b,a_2,\cdots,a_c\}$, then we  can  improve the bound given by
(\ref{4.8}). The details are omitted.

\section{Applications to Reed-Solomon Codes}

Let $D=\{x_1,\cdots, x_n\} \subset {\bf F}_q$ be a subset of
cardinality $|D|=n>0$. For $1\leq k\leq n$, the Reed-Solomon code
$D_{n,k}$ has the codewords of the form
$$(f(x_1), \cdots, f(x_n))\in {\bf F}_q^n,$$
where $f$ runs over all polynomials in ${\bf F}_q[x]$ of degree at
most $k-1$. The minimum distance of the Reed-Solomon code is $n-k+1$
because a non-zero polynomial of degree at most $k-1$ has at most
$k-1$ zeroes. For  $u=(u_1,u_2,\cdots,u_{n})\in {\bf F}_q^n$,  we
can associate a unique polynomial $u(x)\in {\bf F}_q[x]$ of degree
at most $n-1$ such that
$$u(x_i)=u_i,$$
for all $1\leq i\leq n$. The polynomial $u(x)$ can be computed
quickly by solving the above linear system. Explicitly, the
polynomial $u(x)$ is given by the Lagrange interpolation formula
$$u(x) = \sum_{i=1}^n u_i \frac{\prod_{j\not=i}(x-x_j)}{\prod_{j\not=i}(x_i-x_j)}.$$
Define $d(u)$ to be the degree of the associated polynomial $u(x)$
of $u$. It is easy to see that $u$ is a codeword if and only if
$d(u)\leq k-1$.

For a given $u\in {\bf F}_q^n$,  define
$$d(u, D_{n,k}): = \min_{v\in D_{n,k}}d(u,v).$$
The maximum likelihood decoding of $u$ is to find a codeword $v\in
D_{n,k}$ such that $d(u,v) = d(u, D_{n,k})$. Thus, computing
$d(u,D_{n,k})$ is essentially the decision version for the maximum
likelihood decoding problem, which is ${\bf NP}$-complete for
general subset $D\subset {\bf F}_q$. For standard Reed-Solomon code
with $D={\bf F}_q^*$ or ${\bf F}_q$, the complexity of the maximum
likelihood decoding is unknown to be {\bf NP}-complete. This is an
important open problem. It has been shown by Cheng-Wan \cite{CW1,CW2} to
be at least as hard as the discrete logarithm problem.

When $d(u)\leq k-1$,  then $u$ is a codeword and thus
$d(u,D_{n,k})=0$. We shall assume that $k \leq d(u)\leq n-1$. The
following simple result gives an elementary bound for $d(u,
D_{n,k})$.

\begin{thm} Let $u\in {\bf F}_q^n$ be a word such that $k \leq d(u)\leq n-1$.
Then,
$$n-k \geq d(u, D_{n,k}) \geq n-d(u).$$
\end{thm}

{\bf Proof}. Let $v=(v(x_1),\cdots, v(x_n))$ be a codeword of
$D_{n,k}$, where $v(x)$ is a polynomial in ${\bf F}_q[x]$ of degree
at most $k-1$. Then,
$$d(u,v)= n - N_D(u(x)-v(x)),$$
where $N_D(u(x)-v(x))$ denotes the number of zeros of the polynomial
$u(x)-v(x)$ in $D$. Thus,
$$d(u, D_{n,k}) = n - \max_{v\in D_{n,k}} N_D(u(x)-v(x).$$
Now $u(x)-v(x)$ is a polynomial of degree equal to $d(u)$. We deduce
that
$$N_D(u(x)-v(x)) \leq d(u).$$
It follows that
$$d(u, D_{n,k}) \geq n -d(u).$$
The lower bound is proved. To prove the upper bound, we choose a
subset $\{ x_1,\cdots, x_k\}$ in $D$ and let $g(x)=(x-x_1)\cdots
(x-x_k) $. Write
$$u(x)=g(x)h(x) +v(x),$$
where $v(x)\in {\bf F}_q[x]$ has degree at most $k-1$. Then,
clearly, $N_D(u(x)-v(x))\geq k$. Thus
$$d(u, D_{n,k}) \leq n-k.$$
The theorem is proved.

We call $u$ to be a deep hole if $d(u,D_{n,k})=n-k$, that is, the
upper bound in the equality holds. When $d(u)=k$, the upper bound
agrees with the lower bound and thus $u$ must be a deep hole. This
gives $(q-1)q^k$ deep holes. For a general Reed-Solomon code
$D_{n,k}$, it is already difficult to determine if a given word $u$
is a deep hole. In the special case that $d(u)=k+1$, the deep hole
problem is equivalent to the subset sum problem over ${\bf F}_q$
which is {\bf NP}-complete if $p>2$.

For the standard Reed-Solomon code, that is, $D={\bf F}_q^*$ and
thus $n=q-1$, there is the following interesting conjecture of
Cheng-Murray \cite{CM}.

\bigskip
\textbf{Conjecture}  Let $q=p$. For the standard Reed-Solomon code
with $D={\bf F}_p^*$,  the set $\{u\in {\bf F}_p^n \big| d(u)=k\}$
gives the set of all deep holes.

\bigskip
Using the Weil bound, Cheng and Murray proved that their conjecture
is true if $p$ is sufficiently large compared to $k$.

The deep hole problem is to determine when the upper bound in the
above theorem agrees with $d(u, D_{n,k})$. We now examine when the
lower bound $n-d(u)$ agrees with $d(u, D_{n,k})$. It turns out that
the lower bound agrees with $d(u, D_{n,k})$ much more often.
We call $u$ {\bf ordinary}
if $d(u, D_{k,n})=n-d(u)$. A basic problem is then
to determine for a given word $u$,
 when $u$ is ordinary.

Without loss of generality, we can assume that $u(x)$ is monic and
$d(u)=k+m$, $0\leq m \leq n-k$. Let
$$u(x)=x^{k+m} -b_1x^{k+m-1}+\cdots +(-1)^{m} b_{m}x^k +\cdots +(-1)^{k+m} b_{k+m}$$
be a monic polynomial in ${\bf F}_q[x]$ of degree $k+m$. By
definition, $d(u, D_{n,k})=n-(k+m)$ if and only if there is a
polynomial $v(x)\in {\bf F}_q[x]$ of degree at most $k-1$ such that
$$u(x)-v(x) =(x - x_1)\cdots (x-x_{k+m}),$$
with $x_i \in D$ being distinct. This is true if and only if the
system
$$\sum_{i=1}^{k+m} X_i =b_1,$$
$$\sum_{1\leq i_1<i_2 \leq k+m}X_{i_1}X_{i_2} =b_2,$$
$$\cdots,$$
$$\sum_{1\leq i_1<i_2<\cdots <i_m\leq k+m}X_{i_1}\cdots X_{i_m}=b_m.$$
has distinct solutions $x_i\in D$. This explains our motivational
problem in the introduction section.

When $d(u)=k$, then $u$ is always a deep hole. The next non-trivial
case is when $d(u)=k+1$. Using the bound in Theorem \ref{thm1.1}, we obtain
some positive results related to the deep hole problem in the case
$d(u)=k+1$ (i.e., the case $m=1$) if $q-n$ is small.
 When $q-n \leq 1$,  by Corollary \ref{cor2.8}
 we first have the  following simple consequence.
\begin{cor}\label{cor5.2}
Let $q\geq n\geq q-1$ and $q>5$. Let $d(u)=k+1$ with $2<k<q-3$. Then $u$
cannot be a deep hole.
\end{cor}
\begin{proof}
 By the above discussion, $u$ is not a deep hole if and only if
the equation
$$x_1+x_2+\cdots+ x_{k+1}=b$$ always has distinct solutions in $D$
for any $b\in F_q $.  Thus the result follows from Corollary
\ref{cor2.8}.
\end{proof}

 {\bf Remark}.\ \ Similarly, using Theorem \ref{thm1.1}, a simple
asymptotic argument implies that when $q-n$ is a constant, and $d(u)=k+1$ with $2<k<q-3$,
then $u$ cannot be a deep hole for sufficient large $q$.
 Furthermore, for given $q,n$, asymptotic analysis can give
sufficient conditions for $k$ to ensure a degree-$k+1$ word $u$
not being a deep hole.

In the present paper, we studied the case $m=1$ and explored some of
the combinatorial aspects of the problem. In a future article, we
plan to study the case $m>1$ by combining the ideas of the present
papers with algebraic-geometric techniques such as the Weil bound.


\end{document}